\documentclass[12pt]{amsart}
\usepackage{amssymb,amsfonts,amsmath,amscd}
\usepackage{verbatim}
\usepackage[all]{xy}
\usepackage{hyperref}

\newtheorem{Theorem}{\bf Theorem}[section]
\newtheorem{Proposition}[Theorem]{\bf Proposition}
\newtheorem{Lemma}[Theorem]{\bf Lemma}
\newtheorem{Corollary}[Theorem]{\bf Corollary}

\newtheorem{Remark}[Theorem]{\bf Remark}

\newcommand{\fm}{\mbox{$\mathfrak{m}$}}
\newcommand{\fn}{\mbox{$\mathfrak{n}$}}
\newcommand{\fp}{\mbox{$\mathfrak{p}$}}
\newcommand{\fP}{\mbox{$\mathfrak{P}$}}

\newcommand{\bn}{\mbox{$\mathbb{N}$}}

\newcommand{\bq}{\mbox{$\mathbb{Q}$}}

\newcommand{\cm}{\mbox{$\mathcal{M}$}}

\newcommand{\rad}{\mbox{${\rm rad}$}}
\newcommand{\coker}{\mbox{${\rm coker}$}}
\newcommand{\height}{\mbox{${\rm height}$}}
\newcommand{\grade}{\mbox{${\rm grade}$}}
\newcommand{\depth}{\mbox{${\rm depth}$}}
\newcommand{\length}{\mbox{${\rm length}$}}
\newcommand{\rank}{\mbox{${\rm rank}$}}
\newcommand{\kdim}{\mbox{${\rm k}$-${\rm dim}$}}
\newcommand{\prdim}{\mbox{${\rm pr}$-${\rm dim}$}}

\newcommand{\graded}{\mbox{${\bf G}$}}
\newcommand{\adj}{\mbox{${\rm adj}$}}
\newcommand{\Proj}{\mbox{${\rm Proj}$}}
\newcommand{\Tor}{\mbox{${\rm Tor}$}}
\newcommand{\Spec}{\mbox{${\rm Spec}$}}
\newcommand{\Reg}{\mbox{${\rm Reg}$}}

\begin{document}

\title{J. Sally's question and a conjecture of Y. Shimoda}

\author{{\sc Shiro Goto, Liam O'Carroll and Francesc Planas-Vilanova}}

\date{\today}


\begin{abstract}
In 2007, Y.~Shimoda, in connection with a long-standing question of
J.~Sally, asked whether a Noetherian local ring, such that all its
prime ideals different from the maximal ideal are complete
intersections, has Krull dimension at most two. In this paper, having
reduced the conjecture to the case of dimension three, if the ring is
regular and local of dimension three, we explicitly describe a family
of prime ideals of height two minimally generated by three
elements. Weakening the hypothesis of regularity, we find that, to
achieve the same end, we need to add extra hypotheses, such as
completeness, infiniteness of the residue field and the multiplicity
of the ring being at most three. In the second part of the paper we
turn our attention to the category of standard graded algebras. A
geometrical approach via a double use of a Bertini Theorem, together
with a result of A.~Simis, B.~Ulrich and W.V.~Vasconcelos, allows us
to obtain a definitive answer in this setting. Finally, by adapting
work of M.~Miller on prime Bourbaki ideals in local rings, we detail
some more technical results concerning the existence in standard
graded algebras of homogeneous prime ideals with an
``excessive''number of generators.
\end{abstract}


\maketitle

\section{Introduction}\label{introduction}

It is by now a classic result that in a Noetherian local ring, the
existence of a uniform bound on the minimal number of generators of
all its ideals is equivalent to its Krull dimension being at most
1. In 1978, J. Sally (see \cite[p.~52]{sa}) extended this result in
the following way.  Let $(R,\fm,k)$ be a Noetherian local ring. Then
there exists an integer $N\geq 1$ such that the minimal number of
generators $\mu(I)$ of an ideal $I$ is bounded above by $N$, for any
ideal $I$ of $R$ such that $\fm$ is not an associated prime of $I$, if
and only if $\dim(R)$, the Krull dimension of $R$, is at most 2. In
particular, if $\dim(R)\leq 2$, then there exists a bound on the
minimal number of generators of all its prime ideals. She then
remarked that it is an open question whether the converse is true (cf.
[op cit., p.~53]). In other words, if $(R,\fm,k)$ is a Noetherian
local ring such that there exists an integer $N\geq 1$ such that
$\mu(\fp)\leq N$, for all prime ideals $\fp$ of $R$, is then
$\dim(R)\leq 2$?

This question has remained open and not much progress has been made
since that time. In 2007, Y. Shimoda (\cite{shi}) asked whether a
Noetherian local ring, such that all its prime ideals different from
the maximal ideal are complete intersections, has Krull dimension at
most 2. Observe that, if all the primes other than the maximal ideal
are complete intersections, then in particular the cardinalities of
their sets of minimal generators are bounded above by the Krull
dimension of the ring.

Though the question of Y. Shimoda seems easier than that of J. Sally
because its hypothesis is at first sight much stronger, it has proved
to be difficult to answer.

For the sake of simplicity we will call a Noetherian local ring
$(R,\fm,k)$ a Shimoda ring if every prime ideal in the punctured
spectrum is of the principal class, i.e., the minimal number of
generators $\mu(\fp)$ of every prime ideal $\fp$ of $R$, $\fp\neq
\fm$, is equal to $\height(\fp)$.

We first observe that one can reduce the conjecture to the
consideration of a UFD local domain $R$ of Krull dimension at most
3. A positive answer to Shimoda's conjecture then amounts to showing
that either $\dim(R)=2$, or else $\dim(R)=3$ and then exhibiting a
prime ideal of height 2 minimally generated by 3 (or more)
elements. This question is attractive because of the combination of
its simplicity and its seeming difficulty. Indeed, we are not able, up
to now, to produce in full generality a non-complete intersection
prime ideal of height 2 in a UFD local domain of Krull dimension 3.

However, in Section~\ref{mult}, we are able to give some partial
positive answers. Firstly, if the ring $R$ is regular and local of
Krull dimension 3, then we explicitly describe a family of prime
ideals of height 2 minimally generated by 3 elements. These ideals are
determinantal ideals of $2\times 3$-matrices and, in the geometric
case, they are precisely the defining ideals of irreducible affine
space monomial curves.

Next, when trying to weaken the hypothesis of regularity, we find that
we need to add extra hypotheses, such as completeness, infiniteness of
the residue field and the multiplicity of the ring being at most 3. In
this case, if $\dim(R)=3$, we exhibit an ideal of height 2 minimally
generated by 3 elements which has a minimal prime over it which is not
Gorenstein. This will enable us to conclude that a Shimoda ring, with
these extra hypotheses, has Krull dimension at most 2.

In the second part of the paper we turn our attention to the category
of standard graded algebras (with a graded definition of the notion of
a Shimoda ring). In Section~\ref{sgraded} we find that a geometrical
approach via a double use of a Bertini Theorem, together with a result
of A. Simis, B. Ulrich and W.V. Vasconcelos (see \cite{suv}), allows
us to obtain a definitive result. In the final section,
Section~\ref{viamiller}, we sketch an adaptation of M. Miller's
arguments in \cite{miller} to the case of standard graded rings of
interest, that allows us, under somewhat more technical hypotheses, to
produce homogeneous prime ideals requiring an arbitrarily large number
of generators, and also allows us to present a mild generalization of
the core result of Section~\ref{sgraded}.

\section{The Shimoda Conjecture for rings of small 
multiplicity}\label{mult}

Let $(R,\fm,k)$ be a Noetherian local ring, of Krull dimension $d\geq
1$, and $k=A/\fm$, the residue class field of $R$. For the remainder
of this section we fix the following notation: the ring $R$ will be
called a Shimoda ring if every prime ideal in the punctured spectrum
is of the principal class, i.e., the minimal number of generators
$\mu(\fp)$ of every prime ideal $\fp$ of $R$, $\fp\neq \fm$, is equal
to $\height(\fp)$.

\begin{Remark}{\rm Let $(R,\fm,k)$ be a Noetherian local ring of Krull 
dimension $d\geq 1$. If $d=1$, $R$ is Shimoda if and only if $R$ is a
domain. In particular, a Shimoda ring need not be Gorenstein since
$R=k[[t^3,t^4,t^5]]$ is a one-dimensional local domain that is not
Gorenstein (\cite[Ex.~21.11]{e} or \cite[Ex.~18.8,
  p.~152]{m2}). Moreover, the completion of a Shimoda ring is not
necessarily a Shimoda ring, because $A=(k[x,y]/(y^2-x^2-x^3))_{(x,y)}$
is a one-dimensional Noetherian local domain whose completion is not a
domain (\cite[p.~185 and ff.]{e}).  }\end{Remark}

\begin{Remark}\label{ufd} {\rm Let $(R,\fm,k)$ be a Noetherian local 
ring of Krull dimension $d\geq 2$. If $R$ is Shimoda, then $R$ is a
UFD, and the converse holds if $d=2$. Suppose now that $d\geq 2$ and
$R$ is Shimoda. Then $R$ is also Cohen-Macaulay. Indeed, take $\fp$ a
prime ideal with $\height(\fp)=d-1$. Since $\mu(\fp)=\height(\fp)$,
then $\fp$ is generated by a regular sequence $x_{1},\ldots ,x_{d-1}$,
say (see \cite[Remark, p.~203]{da}). Set $y\in\fm$,
$y\not\in\fp$. Since $\fp$ is prime, $x_{1},\ldots, x_{d-1},y$ is a
regular sequence of length $d$.  }\end{Remark}

The purpose of this section is to prove the following result, where
$e(R)$ stands for the multiplicity of $R$ with respect to $\fm$.

\begin{Theorem}\label{shd-m}
Let $(R,\fm,k)$ be a Shimoda ring of Krull dimension $d\geq
2$. Suppose that, in addition, either
\begin{itemize}
\item[$(a)$] $R$ is regular, or
\item[$(b)$] $R$ is complete, $R\supset k$, $k$ is infinite and
  $e(R)\leq 3$.
\end{itemize}
Then $d=2$.
\end{Theorem}

Note in passing that when we weaken the hypothesis from ``regular''
($e(R)=1$) to ``$e(R)\leq 3$'', we need to add the extra hypotheses
that $R$ is complete and contains its infinite residue field. Note
that in either case, $R$ is a local Cohen-Macaulay domain and we use
properties of such domains without further mention. We first show how
to reduce the dimension of $R$.

\begin{Remark}\label{red-d} {\rm
Let $(R,\fm,k)$ be a Shimoda ring of Krull dimension $d\geq 2$. Let
$a\in\fm\setminus\fm^2$. Then $R/aR$ is a Shimoda ring of Krull
dimension $d-1$. Moreover,
\begin{itemize}
\item[$(a)$] If $R$ is regular, then $R/aR$ is
regular. 
\item[$(b)$] If $R$ is complete, then $R/aR$ is complete. If $k$ is
  infinite, then $a$ can be chosen to be a superficial element and
  $e(R/aR)=e(R)$.
\end{itemize}
In particular, in Theorem~\ref{shd-m}, one can suppose that $2\leq
d\leq 3$.  }\end{Remark}

\begin{proof}[Proof of the Remark~\ref{red-d}] 
 Take $a\in\fm\setminus\fm^2$. Given $\fp$, a prime ideal, $\fp\neq
 \fm$, such that $a\in\fp$, then $a\in\fp\setminus\fm\fp$. Therefore
 $\mu(\fp/aR)=\mu(\fp)-1=\height(\fp)-1=\height(\fp/aR)$ and $R/aR$ is
 Shimoda of Krull dimension $d-1$. Clearly, if $R$ is regular, $R/aR$
 is regular, and if $R$ is complete, $R/aR$ is complete. Moreover, if
 $k$ is infinite, there exists a superficial element
 $a\in\fm\setminus\fm^2$ and $e(R/aR)=e(R)$ (see, e.g.,
 \cite[Propositions~8.5.7 and 11.1.9]{sh}). Finally, suppose
 Theorem~\ref{shd-m} true for $2\leq d\leq 3$ and suppose there exists
 a Shimoda ring $R$ of Krull dimension $d\geq 4$. By successively
 factoring out appropriate elements, one would get a Shimoda ring of
 Krull dimension $\overline{d}=3$. But, by Theorem~\ref{shd-m}, in the
 case $2\leq d\leq 3$, one would deduce $\overline{d}=2$, a
 contradiction. Therefore, there can not exist Shimoda rings of Krull
 dimension $d\geq 4$ and one has only to prove Theorem~\ref{shd-m} for
 the case $2\leq d\leq 3$.
\end{proof}

Let us fix now some notations that will hold for the rest of the
section (see \cite{op}).

\begin{Remark}\label{onHN} {\rm 
Let $(R,\fm,k)$ be a Cohen-Macaulay local ring and let $x_1,x_2,x_3$
be a regular sequence in $R$. Take $a=(a_{1},a_{2},a_{3})\in \bn^{3}$
and $b=(b_{1},b_{2},b_{3})\in\bn^{3}$. Let $c=a+b$,
$c=(c_{1},c_{2},c_{3})$. Let $\cm$ be the matrix
\begin{eqnarray*}
\cm=\left(\begin{array}{ccc}x_{1}^{a_{1}}&x_{2}^{a_{2}}&x_{3}^{a_{3}}
  \\ x_{2}^{b_{2}}&x_{3}^{b_{3}}&x_{1}^{b_{1}}\end{array}\right), 
\end{eqnarray*}
and $v_{1}=x_{1}^{c_{1}}-x_{2}^{b_{2}}x_{3}^{a_{3}}$,
$v_{2}=x_{2}^{c_{2}}-x_{1}^{a_{1}}x_{3}^{b_{3}}$ and
$D=x_{3}^{c_{3}}-x_{1}^{b_{1}}x_{2}^{a_{2}}$, the $2\times 2$ minors
of $\cm$ up to a change of sign. Consider
$I=I_{2}(\cm)=(v_{1},v_{2},D)$, the determinantal ideal generated by
the $2\times 2$ minors of $\cm$. Then $I$ is a non-Gorenstein
height-unmixed ideal of height two, minimally generated by three
elements. For simplicity, $I$ will be called the HN (from Herzog
Northcott) ideal associated to $x_1,x_2,x_3$ and $a,b\in\bn^3$. We
will also set $m_1=c_{2}c_{3}-a_{2}b_{3}$,
$m_2=c_{1}c_{3}-a_{3}b_{1}$, $m_3=c_{1}c_{2}-a_{1}b_{2}$ and
$m(a,b)=(m_1,m_2,m_3)\in\bn^3$. Note that $m_i\geq 3$.  }\end{Remark}

\begin{proof}[Proof of Remark~\ref{onHN}] That $I$ is height-unmixed 
of height 2 follows from \cite[4.2 and 2.2]{op}. That $I$ is minimally
generated by 3 elements and is non-Gorenstein follows from
\cite[6.1]{op} and its proof, where a non-symmetric minimal resolution
of $R/I$ is shown.
\end{proof}

The proof of Theorem~\ref{shd-m} is divided into two parts. We first
state the regular case.

\begin{Proposition}\label{shd-reg}
Let $(R,\fm,k)$ be a regular local ring of Krull dimension $3$. Let
$x_1,x_2,x_3$ be a regular system of parameters in $R$ and
$a,b\in\bn^3$. Let $I$ be the HN ideal associated to $x_1,x_2,x_3$ and
$a,b$. If $\gcd(m(a,b))=1$, then $I$ is prime. In particular, $R$ is not
Shimoda.
\end{Proposition}

Taking $a=(1,1,1)$ and $b=(2,1,1)$ in Proposition~\ref{shd-reg}, one
obtains the following result.

\begin{Corollary}\label{natural-candidate}
Let $(R,\fm,k)$ be a regular local ring of Krull dimension $3$. Let
$x_1,x_2,x_3$ be a regular system of parameters in $R$. Then
$I=(x_{1}^{3}-x_{2}x_{3},x_{2}^{2}-x_{1}x_{3},x_{3}^{2}-x_{1}^{2}x_{2})$
is a height two prime ideal minimally generated by three elements.
\end{Corollary}

For the case of small multiplicity, we have the following. As
above, we consider the HN ideal associated to $a=(1,1,1)$ and
$b=(2,1,1)$.

\begin{Proposition}\label{shd-mc}
Let $(R,\fm,k)$ be a complete Gorenstein local domain of Krull
dimension $3$. Suppose that, in addition, $R\supset k$, $k$ is
infinite and $e(R)\leq 3$. Let $(x_1,x_2,x_3)R$ be a minimal reduction
of $\fm$. Then there is a minimal prime over
$I=(x_{1}^{3}-x_{2}x_{3},x_{2}^{2}-x_{1}x_{3},x_{3}^{2}-x_{1}^{2}x_{2})$
which is not Gorenstein. In particular, $R$ is not Shimoda.
\end{Proposition}

Let us give now a proof of Theorem~\ref{shd-m} using
Propositions~\ref{shd-reg} and \ref{shd-mc}.

\begin{proof}
[Proof of Theorem~\ref{shd-m} using Propositions~\ref{shd-reg},
  \ref{shd-mc}] Let $R$ be a Shimoda ring of Krull dimension $d\geq 2$
as in Theorem~\ref{shd-m}. By Remark~\ref{red-d}, one can suppose that
$2\leq d\leq 3$. If $R$ is regular, by Proposition~\ref{shd-reg},
$d\neq 3$. Suppose now that $R$ is complete, $R\supset k$, $k$
infinite and $e(R)\leq 3$. By Remark~\ref{ufd}, $R$ is Cohen-Macaulay,
and since $R$ is complete, it admits a canonical module (see, e.g.,
\cite[3.3.8]{bh}). By Remark~\ref{ufd} again, $R$ is a UFD. Hence $R$
is Gorenstein (see, e.g., \cite[3.3.19]{bh}). Therefore, by
Proposition~\ref{shd-mc}, $d\neq 3$.
\end{proof}

Before proving Proposition~\ref{shd-reg}, we note the following
reasonably elementary fact.

\begin{Lemma}\label{length-quotient}
Let $(S,\fn)$ be a Cohen-Macaulay local ring of Krull dimension
$2$. Let $x,y$ be a system of parameters of $S$. Let $p,q,r,s\in\bn$,
with $1\leq r<p$ and $1\leq s<q$.  Then
\begin{eqnarray*}
\length_S(S/(x^p,y^q,x^ry^s))=[pq-(p-r)(q-s)]\cdot \length_S(S/(x,y)).
\end{eqnarray*}
\end{Lemma}

\begin{proof}
Let $I=(x,y)$. Since $I$ is generated by a regular sequence, there is
a natural graded isomorphism $(S/I)[X,Y]\cong \graded(I)$, sending $X$
to $x+I^{2}$ and $Y$ to $y+I^{2}$, between the polynomial ring in two
indeterminates $X,Y$ over $S/I$ and the associated graded ring
$G=\graded(I)=\oplus _{n\geq 0}I^n/I^{n+1}$ of $I$. If $z\in S$, let
$z^{*}$ denote its initial form in $G$, i.e., $z^{*}=z+I^{m}$, where
$z\in I^{m-1}\setminus I^{m}$, and $0^{*}=0$. In particular,
$(x^{*})^{n}=(x^{n})^{*}$ and $(y^{*})^{m}=(y^{m})^{*}$, for all
$n,m\geq 1$. If $J$ is an ideal of $S$, let $J^{*}$ denote the
homogeneous ideal of $G$ generated by all the initial forms of
elements of $J$. One has $J^{*}_{n}=(J\cap I^{n})+I^{n+1}/I^{n+1}$ for
$n\geq 0$ (see, e.g., \cite{vv}).

Now take $J=(x^p,y^q,x^ry^s)\subset I$. Observe that, for $n\geq p+q$,
$I^{n}\subseteq J$ and $(J^{*})_n=G_n$. For each $n\geq 0$, consider
the short exact sequences of $S$-modules:
\begin{eqnarray*}
0\to (J+I^{n+1})/J\to (J+I^{n})/J\to (J+I^{n})/(J+I^{n+1})\to 0.
\end{eqnarray*}
Hence
$\length_{S}(S/J)=\sum_{n=0}^{p+q-1}\length_S((J+I^{n})/(J+I^{n+1}))$. But
\begin{eqnarray*}
\frac{J+I^{n}}{J+I^{n+1}}\cong \frac{I^{n}}{I^{n}\cap
  (J+I^{n+1})}\cong \frac{I^{n}}{(J\cap I^{n})+I^{n+1}}\cong
\frac{G_{n}}{J^{*}_{n}}.
\end{eqnarray*}
Thus 
\begin{eqnarray*}
\length_S(S/J)=\sum_{n=0}^{p+q-1}\length_S(G_{n}/J^{*}_{n})=\sum_{n\geq
  0}\length_{S}(G_{n}/J^{*}_{n})=\length_S(G/J^{*}).
\end{eqnarray*}

Let $L\subseteq J^{*}$ be the ideal generated by the initial forms of
$x^{p}$, $y^{q}$ and $x^{r}y^{s}$ in $G$. Since the regular sequence
$x,y$ is permutable, by \cite[Remark~4(4)]{ks}, $L=J^{*}$.  Hence
$\length_S(G/J^{*})= \length_S(G/L)<\infty$. Through the isomorphism
$(S/I)[X,Y]\cong \graded(I)$, one sees that $G/L$ is isomorphic to the
free $S/I$-module with basis $X^{i}Y^{j}$, where $(i,j)\in \{ 0,\ldots
,p-1\}\times \{0,\ldots ,s-1\}$ or $(i,j)\in \{ 0,\ldots ,r-1\}\times
\{s,\ldots ,q-1\}$.  Therefore $\length_{S}(G/L)=[pq-(p-r)(q-s)]\cdot
\length_S(S/I)$.
\end{proof}

\begin{Remark}\label{standard-base} {\rm The result of K. Kiyek and 
J. St\"uckrad (cf. \cite[Remark~4(4)]{ks}) that we have used in the
proof of Lemma~\ref{length-quotient}, which deals only with the case
where the regular sequence at hand is permutable, says precisely that
$x^{p},y^{q},x^{r}y^{s}$ is an $I$-standard base of
$J=(x^{p},y^{q},x^{r}y^{s})$ (see, e.g., \cite[\S~13]{hio}). In this
respect, one can prove the following fact, which uses a generalization
of a classic theorem of Rees
(cf. \cite[Theorem~1.1.7]{bh}).}\end{Remark}

\begin{Theorem}\label{rees-theorem}
Let $R$ be a commutative ring and let $\underline{z}=z_1,\ldots ,z_n$
be an $R$-sequence. Set $I=(z_1,\ldots ,z_n)$. Let
$\underline{Z}=Z_1,\ldots ,Z_n$ be indeterminates over $R$.
\begin{itemize}
\item[$(a)$] If $F(\underline{Z})\in R[\underline{Z}]$ is homogeneous
  of degree $d$ and $F(\underline{z})\in I^{d+i}$, then
  $F(\underline{Z})\in I^{i}[\underline{Z}]$.
\item[$(b)$] If $m_1,\ldots ,m_s$ is a finite set of monomials in
  $\underline{z}$, then $m_1,\ldots ,m_s$ is an $I$-standard base of
  $J=(m_1,\ldots,m_s)$.
\end{itemize}
\end{Theorem}

\begin{proof}
If $i=0$, the result is trivial. Rees' Theorem is the case
$i=1$. Suppose $i>1$ and that we have established the result in the
case where $i$ is replaced by $i-1$. Suppose then that
$F(\underline{Z})\in R[\underline{Z}]$ is homogeneous of degree $d$
and $F(\underline{z})\in I^{d+i}$. Then $F(\underline{z})\in
I^{d+i-1}$, so by the inductive hypothesis the coefficients of $F$ lie
in $I^{i-1}$. Thus, with $\lambda=(\lambda_1,\ldots,\lambda_n)$,
\begin{eqnarray*}
F(\underline{Z})=\sum_{\lambda_1+\ldots
  +\lambda_n=d}a_{\lambda}Z_1^{\lambda_1}\ldots Z_n^{\lambda_n},
\end{eqnarray*}
for $a_{\lambda}\in I^{i-1}$. Hence, for each such $\lambda$,
\begin{eqnarray*}
a_{\lambda}=\sum_{\mu_1+\ldots
  +\mu_n=i-1}b_{\lambda,\mu}z_1^{\mu_1}\ldots z_n^{\mu_n},
\end{eqnarray*}
for elements $b_{\lambda,\mu}$ in $R$, with $\mu=(\mu_1,\ldots
,\mu_n)$ and $\mu_1+\ldots +\mu_n=i-1$. Let $\xi=(\xi_1,\ldots
,\xi_n)=\lambda+\mu$, so that $\xi_1+\ldots +\xi_n=d+i-1$, and let
\begin{eqnarray*}
H(\underline{Z})=\sum_{(\lambda_1+\ldots
  +\lambda_n=d)}\sum_{(\mu_1+\ldots
  +\mu_n=i-1)}b_{\lambda,\mu}Z_1^{\xi_1}\ldots Z_n^{\xi_n}\in
R[\underline{Z}].
\end{eqnarray*}
Then $H$ is homogeneous of degree $d+i-1$ and
$H(\underline{z})=F(\underline{z})\in I^{d+i}$. By Rees' Theorem, the
coefficients $b_{\lambda,\mu}$ of $H$ lie in $I$, and it follows that
each $a_{\lambda}$ lies in $I^{i}$, which proves $(a)$.

Let $m_1,\ldots ,m_s$ be monomials in $\underline{z}$, where
$m_i=z_1^{\alpha_{1,i}}\ldots z_n^{\alpha_{n,i}}$, for $i=1,\ldots
,s$. Let $\alpha_i=\alpha_{1,i}+\ldots +\alpha_{n,i}$ be the
``degree'' of $m_i$, which is well-defined due to the standard
isomorphism $(R/IR)[Z_1,\ldots ,Z_n]\to \graded(I)=\oplus_{d\geq
  0}I^{d}/^{d+1}$, sending $Z_i$ to $z_i+I^2$. Let $\delta_1<\ldots
<\delta_t$ be the degrees occurring in
$\{\alpha_{1},\ldots,\alpha_{s}\}$, $t\geq 1$. Let $J_{i}$ denote the
ideal of $R$ generated by those monomials among $m_{1},\ldots,m_{s}$
that have degree $\delta_{i}$. We must show that, for $d\geq 0$,
\begin{equation*}
J\cap I^{d}=m_{1}I^{d-\alpha_{1}}+\ldots +m_{s}I^{d-\alpha_{s}},
\end{equation*}
or equivalently,
\begin{equation}\label{c-stand-base}
J\cap I^{d}=J_{1}I^{d-\delta_{1}}+\ldots +J_{t}I^{d-\delta_{t}},
\end{equation}
understanding that $I^{r}=R$ whenever $r\leq 0$ and noting that $I$
and $J$ are not comaximal (see \cite[p.~94, lines~23-24]{vv}). We prove
(\ref{c-stand-base}) by induction on $t\geq 1$.

Item $(a)$ easily yields the case where $t=1$, i.e., where all the
monomials have the same degree. So we suppose that $t\geq 2$ and that
(\ref{c-stand-base}) holds when $t$ is replaced by $t-1$. Suppose
first of all that $d\leq \delta _{t}$. Then, by induction,
\begin{equation*}
J\cap I^{d}=(J_{1}+\ldots +J_{t})\cap I^{d}=((J_{1}+\ldots
+J_{t-1})\cap I^{d})+J_{t}=J_{1}I^{d-\delta_{1}}+\ldots
+J_{t-1}I^{d-\delta_{t-1}}+J_{t}I^{d-\delta_{t}},
\end{equation*}
since $I^{d-\delta_{t}}=R$ in this case. Finally, suppose on the other
hand that $d>\delta_{t}$. Then
\begin{equation*}
J\cap I^{d}=J\cap I^{\delta_{t}}\cap I^{d}=(J_{1}+\ldots +J_{t})\cap
I^{\delta_{t}}\cap I^{d}=(J_{1}I^{\delta_{t}-\delta_{1}}+\ldots +J_{t})\cap I^{d},
\end{equation*}
by the case just considered. But the ideal
$J_{1}I^{\delta_{t}-\delta_{1}}+\ldots +J_{t}$ is generated by
monomials in $\underline{z}$ of degree $\delta_{t}$, so by the
case $t=1$,
\begin{equation*}
(J_{1}I^{\delta_{t}-\delta_{1}}+\ldots +J_{t})\cap
  I^{d}=(J_{1}I^{\delta_{t}-\delta_{1}}+\ldots
  +J_{t})I^{d-\delta_{t}}=J_{1}I^{d-\delta_{1}}+\ldots
  +J_{t}I^{d-\delta_{t}},
\end{equation*}
as required.
\end{proof}

To prove Proposition~\ref{shd-reg}, we are going to use formulae from
the theory of multiplicities (see, e.g., \cite[Chapter~4, Section
  6]{bh} and \cite[Chapter~11, Sections 1 and 2]{sh}).

\begin{proof}[Proof of Proposition~\ref{shd-reg}]
Let $\widehat{R}$ be the completion of $R$. Since $R$ is regular and
local with maximal ideal $\fm$ generated by the regular system of
parameters $x_1,x_2,x_3$, then $\widehat{R}$ is a regular local ring
with maximal ideal $\fm\widehat{R}=(x_1,x_2,x_3)\widehat{R}$,
generated by the regular system of parameters $x_1,x_2,x_3$
(considered in $\widehat{R}$). Let $I=(v_1,v_2,D)$ be the HN ideal
associated to $x_1,x_2,x_3$ and $a,b\in\bn^3$ in $R$. Then
$I\widehat{R}=(v_1,v_2,D)\widehat{R}$ is the HN ideal associated to
$x_1,x_2,x_3$ and $a,b\in\bn^3$, regarded in $\widehat{R}$. Since
$I=I\widehat{R}\cap R$, if we prove that $I\widehat{R}$ is prime, then
$I$ will be prime too. Therefore we can suppose that $R$ is complete.

Until further notice we assume only that $R$ is a complete
Cohen-Macaulay local ring of dimension $3$ and that $x_1,x_2,x_3$
generates a reduction of $\fm$. We do this so that we can avoid
repeating core aspects of the argument when we come to prove
Proposition~\ref{shd-mc}.

Take $\fp$ any associated prime of $I$, hence of height 2.  Let
$D=R/\fp$. Then $D$ is a one-dimensional Noetherian domain. Let $V$ be
its integral closure in its quotient field $K$. Since $R$ is a
complete local ring, $R$ is a Nagata ring (see, e.g.,
\cite[p.~234]{m1}) and hence $V$ is a finite $D$-module and, in
particular, $V$ is Noetherian as a ring. Since $D$ is a Noetherian
complete local domain, its integral closure $V$ is a local ring (see,
e.g., \cite[Corollary~6.1, p.~116]{di}). Therefore, $V$ is a
one-dimensional Noetherian local integrally closed ring, hence a DVR.

Let $\nu$ denote the corresponding valuation on $K$. By abuse of
notation, let $x_i$ denote the image in $V$ of each $x_i$ and set
$\nu_i=\nu(x_i)$. In $V$, $x_1^{c_1}=x_2^{b_2}x_3^{a_3}$,
$x_2^{c_2}=x_1^{a_1}x_3^{b_3}$ and
$x_3^{c_3}=x_1^{b_1}x_2^{a_2}$. Applying $\nu$ to these equalities,
one gets the following system of equations:
\begin{eqnarray*}
\left. \begin{array}{ccc}c_{1}\nu_{1}&=&b_{2}\nu_{2}+a_{3}\nu_{3}\\
c_{2}\nu_{2}&=&a_{1}\nu_{1}+b_{3}\nu_{3}\\
c_{3}\nu_{3}&=&b_{1}\nu_{1}+a_{2}\nu_{2}\end{array} \right\},
\end{eqnarray*}
the third equation (say) being the sum of the first two. So we reduce
to a system of two linearly independent equations, considered over
$\bq$, whose solution is $(\nu_1,\nu_2,\nu_3)=q\cdot m(a,b)$, for some
$q\in\bq$, $q=u/v$, $u,v\in\bn$ (see \cite[4.4]{op}).

Set $l=\gcd(\nu_1,\nu_2,\nu_3)\in\bn$. Since
$v(\nu_1,\nu_2,\nu_3)=u(m_1,m_2,m_3)$ and $\gcd(m_1,m_2,m_3)=1$, then
$vl=u$ and $l=q\in\bn$. Hence $\nu_i=lm_i$.

Clearly $x_1R+I=(x_1,x_2^{c_2},x_3^{c_{3}},x_2^{b_2}x_3^{a_3})$ is an
$\fm$-primary ideal of $R$. Set $S=R/x_1R$, and (by abuse of notation)
consider $x_2,x_3$ as a regular sequence in $S$ and a system of
parameters of $S$. Since $R/(x_1R+I)\cong
S/(x_2^{c_2},x_3^{c_{3}},x_2^{b_2}x_3^{a_3})S$ (and $x_1R\subset
\mbox{Ann}_{R}(S)$), then $\length_R(R/(x_1R+I))=
\length_{S}(S/(x_2^{c_2},x_3^{c_{3}},x_2^{b_2}x_3^{a_3})S)$, which by
Lemma~\ref{length-quotient}, is equal to $m_1\cdot
\length_S(S/(x_2,x_3)S)$. But $S/(x_2,x_3)S\cong R/(x_1,x_2,x_3)$, so
\begin{eqnarray*}
\length_S(S/(x_2,x_3)S)=\length_R(R/(x_1,x_2,x_3))= e_R(x_1,x_2,x_3),
\end{eqnarray*}
because $R$ is Cohen-Macaulay of Krull dimension $d=3$ and
$(x_1,x_2,x_3)$ is an $\fm$-primary ideal
(\cite[11.1.10]{sh}). Moreover, as $(x_1,x_2,x_3)$ is a reduction of
$\fm$, then $e_R(x_1,x_2,x_3)=e(R)$ (\cite[1.2.5 and 11.2.1]{sh}).

Since $x_1R+I\subseteq x_1R+\fp$, then $\length_R(R/(x_1R+I))\geq
\length_R(R/(x_1R+\fp))$. But
\begin{eqnarray*}
\length_R(R/(x_1R+\fp))= \length_{R/\mathfrak{p}}((R/\fp)/(x_{1}\cdot
R/\fp))=\length_D(D/x_1D).
\end{eqnarray*}

Observe that $x_2^{c_{2}},x_{3}^{c_{3}}\in x_1D$ and so $x_1D$ is an
$\fm/\fp$-primary ideal of the one-dimensional Cohen-Macaulay local
domain $(D,\fm/\fp,k)$. By \cite[11.1.10]{sh},
$\length_D(D/x_1D)=e_D(x_1;D)$.

On the other hand, $V$ is a finite Cohen-Macaulay $D$-module of
$\rank_D(V)=1$, and so $e_D(x_1;D)=e_{D}(x_1;D)\cdot \rank
_D(V)=\length_D(V/x_1V)$ (\cite[4.6.11]{bh}).

Set $r=[k_V:k_D]$, the degree of the extension of the residue fields
of $V$ and of $D$. Then $\length_D(V/x_1V)=r\cdot
\length_V(V/x_1V)$. Finally, since $V$ is a DVR,
$\length_V(V/x_1V)=\nu(x_1)=\nu_1=lm_1$.

Therefore, putting together all the (in)equalities, we have:
\begin{eqnarray}\label{inequality}
m_1\cdot e(R)=\length_R(R/(x_1R+I))\geq
\length_R(R/(x_1R+\fp))=rlm_1.
\end{eqnarray}

Observe that in all this reasoning, we have only used that $R$ is a
complete Cohen-Macaulay local ring of dimension $3$ and that
$x_1,x_2,x_3$ generates a reduction of $\fm$.

Now, using that $R$ is regular, we have $e(R)=1$ and we deduce from
(\ref{inequality}) that 
\begin{eqnarray*}
\length_R(R/(x_1R+I))= \length_R(R/(x_1R+\fp)).
\end{eqnarray*}
On tensoring the exact sequence $0\to \fp/I\to R/I\to R/\fp\to 0$ by
$R/x_1R$, one obtains the exact sequence
\begin{eqnarray*}
0\to L/x_1L\to R/(x_1R+I)\to R/(x_1R+\fp)\to 0,
\end{eqnarray*}
where $L=\fp/I$, because $x_1$ is not in $\fp$ so $x_1R\cap
\fp=x_1\fp$. Then, from the equality $\length_R(R/(x_1R+I))=
\length_R(R/(x_1R+\fp))$, one deduces that $\length_R(L/x_1L)=0$ and
hence $L=x_1L$, which by the Lemma of Nakayama implies that $L=0$ and
$I=\fp$.
\end{proof}

Now we turn to the proof of Proposition~\ref{shd-mc}.

\begin{proof}[Proof of Proposition~\ref{shd-mc}] 
First observe that since $k$ is infinite, there exist $x_1,x_2,x_3$ in
$\fm$ such that $(x_1,x_2,x_3)$ is a minimal reduction of $\fm$. In
particular, $x_1,x_2,x_3$ is a system of parameters in $R$. Since
$(R,\fm,k)$ is a complete Noetherian local domain and $R\supset k$,
there exists a $k$-algebra homomorphism $\varphi:k[[X_1,X_2,X_3]]\to
R$, from the power series ring in the three indeterminates
$X_1,X_2,X_3$ over $k$ to $R$, with $\varphi(X_i)=x_i$, and such that
if $S=\mbox{im}(\varphi)=k[[x_1,x_2,x_3]]$, then $S\cong
k[[X_1,X_2,X_3]]$ via $\varphi$, $S$ is a complete regular local ring,
and $R$ is a finite extension of $S$ (see, e.g., \cite[29.4]{m2} and
its proof). Let $\psi:k[[X_1,X_2,X_3]]\to k[[t]]$ be defined as
$\psi(X_1)=t^{3}$, $\psi(X_2)=t^{4}$ and $\psi(X_3)=t^{5}$, where $t$
is an indeterminate over $k$. Then
$\ker(\psi)=(X_{1}^{3}-X_{2}^{}X_{3}^{},
X_{2}^{2}-X_{1}^{}X_{3}^{},X_{3}^{2}-X_{1}^{2}X_{2}^{})$ (see e.g.,
\cite[7.8]{op}, where a proof for the polynomial case is given). Set
$J=\varphi(\ker(\psi))$, which is a prime ideal of $S$ and set
$I=JR=(x_{1}^{3}-x_{2}^{}x_{3}^{}, x_{2}^{2}-x_{1}^{}x_{3}^{},
x_{3}^{2}-x_{1}^{2}x_{2}^{})$, which is the HN ideal of $R$ associated
to $x_1,x_2,x_3$ and $a=(1,1,1)$ and $b=(2,1,1)$.

Take now any minimal prime $\fp$ over $I$ and suppose that $\fp$ is
Gorenstein. We will reach a contradiction.

Since $JR=I\subseteq \fp$ and $S\subset R$ is an integral extension,
then $J\subseteq I\cap S\subseteq \fp\cap S$ and $S/(\fp\cap S)\subset
R/\fp$ is an integral extension. Hence $\fp\cap S$ is a prime ideal of
$S$ of height 2 and $J=\fp\cap S$. Set $A=S/J\cong
k[[t^{3},t^{4},t^{5}]]\subset k[[t]]$ and $D=R/\fp$. By the
Auslander-Buchsbaum formula, $R$ is a free $S$-module of
$\rank_S(R)=e(S)\cdot \rank_S(R)=e(R)=:e$ (see, e.g.,
\cite[4.6.9]{bh}). By base change, it follows that $R/I$ is a free
$A$-module of rank $e$. Hence $D$ is a torsion-free $A$-module of rank
$e^{\prime}$, where $1\leq e^{\prime}\leq e$.

We claim that $e^{\prime}>1$. Indeed, suppose that
$e^{\prime}=1$. Then the quotient field of $D$ can be identified with
the quotient field of $A$. By \cite[Ex.~21.11]{e}, the quotient field
of $A$ is $k((t))$ and the integral closure $V$ of $A$ is
$k[[t]]$. Since $A$ and $D$ have the same quotient field with
$A\subset D$ a finite and hence integral extension, $V$ is also the
integral closure of $D$. By \cite[op. cit.]{e}, the conductor
$A:_AV$ of $V$ into $A$ equals $(t^3,t^4,t^5)A=\fm_A$, the maximal
ideal of $A$. Since $A:_AV$ is also an ideal in $V$ and $t^3\in
A:_AV$, we therefore have
\begin{eqnarray*}
t^3V\subseteq A:_AV=\fm_A=(t^3,t^4,t^5)A\subseteq (t^3,t^4,t^5)V=t^3V.
\end{eqnarray*}
Therefore $t^3V=A:_AV=\fm_A$. Now $x_1D$ is a reduction of the maximal
ideal $\fm_D$ of $D$, since $(x_1,x_2,x_3)R$ is a reduction of
$\fm$. Hence the integral closure $\overline{x_1D}$ of the ideal
$x_1D$ equals $\fm_D$, so $x_1V=\overline{(x_1D)}V=\fm_DV$ (see, e.g.,
\cite[6.8.1]{sh}). Since $x_1V=t^3V$ and $A\subset D$ is integral, it
follows that
\begin{eqnarray*}
\fm_D\subseteq \fm_DV=t^3V=\fm_A\subseteq \fm_D.
\end{eqnarray*}
Hence $\fm_D=\fm_DV=t^3V$. In particular, $D\neq V$ and $\fm_D$ is
also an ideal of $V$. Hence $\fm_D\subseteq D:_DV\subsetneq D$ so
$\fm_D=D:_DV=t^3V$.  Using the fact that $D$ is Gorenstein and
\cite[12.2.2]{sh}, one has
\begin{eqnarray*}
&&2=2\cdot \length_D(D/\fm_D)=2\cdot \length_D(D/(D:_DV))=\\&&
  \length_D(V/(D:_DV))=\length_D(V/t^3V)=r\cdot \length_V(V/t^3V)=3r,
\end{eqnarray*}
where $r=[k_V:k]$, which is a contradiction. Thus $e^{\prime}\geq 2$.

Now, $R/I$ and $A$ are local rings with the same residue field. Hence,
since $x_1A$ is an $\fm_A$-primary ideal of the Noetherian local
domain $A$ and $\rank_A(D)=e^{\prime}$, by \cite[11.2.6]{sh}, and
regarding $D$ as an $R/I$-module, we have the following:
\begin{eqnarray*}
e_{R/I}(x_1\cdot R/I;D)=e_A(x_1A;D)=e_A(x_1A;A)\cdot
\rank_A(D)=3e^{\prime}\geq 6.
\end{eqnarray*}
Analogously, since $\rank_A(R/I)=e\leq 3$ by hypothesis, then
\begin{eqnarray*}
e_{R/I}(x_1\cdot R/I;R/I)=e_A(x_1A;R/I)=e_A(x_1A;A)\cdot
\rank_A(R/I)=3e\leq 9.
\end{eqnarray*}
But, by the associativity formula (see, e.g., \cite[11.2.4]{sh}),
letting $\fp$ vary through the minimal primes in $R$ over $I$, we
have
\begin{eqnarray*}
&&9\geq e_{R/I}(x_1;R/I)=\sum_{\mathfrak{p}}e_{R/I}(x_1;R/\fp)\cdot
  \length_{R_{\mathfrak{p}}}(R_{\mathfrak{p}}/IR_{\mathfrak{p}})\geq
  6\cdot (\mbox{the number of such }\fp).
\end{eqnarray*}
Hence the number of such $\fp$ equals 1 and for this unique $\fp$,
$\length_{R_{\mathfrak{p}}}(R_{\mathfrak{p}}/IR_{\mathfrak{p}})=1$. Therefore
$I=\fp$, which is a contradiction, since $I$ is not Gorenstein and
$\fp$ is Gorenstein.
\end{proof}

\section{The Shimoda Conjecture in the setting of a standard graded 
algebra}\label{sgraded}

In connection with the Shimoda property for affine rings, we note the
following result.

\begin{Proposition} 
Let $S$ be an affine domain of Krull dimension $d$ at least $3$. Then
$S$ contains a prime ideal $P$ that requires more than $\height(P)$
generators.
\end{Proposition}

\begin{proof}
Since $S$ is an excellent domain, the regular locus $\Reg(S)$ is a
non-empty Zariski-open subset of $\Spec(S)$. Hence there exists a
non-zero element $s\in S$ such that $S_{s}$ is a regular ring. Since
$S$ is a Hilbert ring with $(0)$ a prime ideal, there exists a maximal
ideal $M\in S$ such that $s\notin M$, so $S_{M}$ is a regular local
ring. Now $\height(M)=d\geq 3$ and so, by Theorem~\ref{shd-m}$(a)$,
and its proof, there exists $\fp\in\Spec(S_{M})$ having dimension 1
such that $\mu(\fp)=d$. Set $P=\fp\cap S$. Since $\fp=P_{M}$ it
follows that $\height(P)=d-1$, yet $P$ requires at least $d$
generators.
\end{proof}

\begin{Remark}{\rm 
Note however that if $S$ is also a standard graded algebra, the
preceding result does not provide any information as to whether the
resulting prime ideal $P$ is or is not homogeneous.  
}\end{Remark}

We now consider a different but analogous set-up to the local case
considered in Section~\ref{mult}, influenced by the well-known
similarities between the theories of local and standard graded
rings. So let $A=k[x_{1},\ldots ,x_{n}]$ be a standard graded algebra
over the field $k$, i.e., $A$ is graded by the non-negative integers
$\bn_{0}$ with $k$ sitting in degree $0$ and each of the $x_{i}$
having degree 1. For simplicity we suppose that $k$ is infinite (in
fact, eventually we will suppose that $k$ has characteristic 0 or is
algebraically closed). Let $\cm:=(x_{1},...,x_{n})A$ denote the
irrelevant ideal of $A$, and we take $n$ to be the minimal number of
homogeneous generators $\mu (\cm)$ of $\cm$ (see \cite[1.5.15(a)]{bh},
noting that $\mu(\cm)=\mu (\cm_{\mathcal{M}})$).  We fix this
notation. We also suppose that $A$ is a \textit{Shimoda ring in the
  graded sense}, or a \textit{gr-Shimoda ring} for short, meaning that
each relevant homogeneous prime ideal $\fp$ of $A$ is generated by
$\height(\fp)$ elements, i.e., is of the principal class: note that
these generators can be chosen to be homogeneous (once again, see
\cite[1.5.15(a)]{bh}).

Finally, we suppose throughout that $A$ has Krull dimension at least
2.

\begin{Proposition}
Suppose that the gr-Shimoda ring $A$ has Krull dimension $d$ (with
$d\geq 2$). Then $A$ is a Gorenstein domain that satisfies the Serre
condition $(R_{d-1})$.
\end{Proposition}

\begin{proof} On taking $\fp$ to be a minimal prime (necessarily 
homogeneous) in the above condition, we see that $\fp=(0)$ so that
$A$ is a domain. Next let $\fp$ be a homogeneous prime that is maximal
with respect to the property of being relevant. By Davis' result
\cite[Remark p.~203]{da}, $\fp$ is generated by a regular sequence of
length $\height(\fp)$. It follows that $\cm$ contains a regular
sequence of length $\height(\cm)$. Hence $A$ is Cohen-Macaulay (see
\cite[Ex.~2.1.27(c)]{bh}).

Next, since $A$ is a homomorphic image of a polynomial ring (with the
standard grading) by a homogeneous morphism of degree $0$, $A$ has a
graded canonical module $C$ (see \cite[3.6.10 and 3.6.12(b)]{bh}). By
\cite[Ex.~21.18(b)]{e}, we may take $C$ to be a homogeneous ideal of
$A$. If $C=A$, then $A$ is Gorenstein (\cite[3.6.11]{bh}). Suppose on
the other hand that $C$ is a proper ideal in $A$. By \cite[3.6.9 and
  3.3.18]{bh}, $C$ is an unmixed ideal of height 1 whose (necessarily
homogeneous) associated primes are therefore principal, by the
gr-Shimoda property. Let $\fp$ be any one of these associated primes,
with $\fp=(p)$ say. Then each $A_{\mathfrak{p}}$ is a DVR, so that
each localitzation $C_{\mathfrak{p}}$ is a principal ideal,
$C_{\mathfrak{p}}=p^{t_{p}}A_{\mathfrak{p}}$, say. It is easily seen
that $C=(\prod_{p}p^{t_{p}})A$, since the latter ideal has the same
associated primes as $C$ and agrees with $C$ locally at each of these
primes. Hence $C$ is principal. That $A$ is Gorenstein now follows
from \cite[3.6.11]{bh}.

Finally, to show that $A$ satisfies $(R_{d-1})$, we examine
$A_{\mathfrak{p}}$ for each prime ideal $\fp$ of height $h$, with
$h\leq d-1$. Clearly we need only consider the case where $\fp$ is
non-homogeneous, so that $h>0$. By the first paragraph of the proof of
\cite[1.5.8]{bh}, or by \cite[(5.1) Lemma, (b)]{f}, we see that for
such $\fp$, $A_{\mathfrak{p}}$ is a regular local ring, and the result
follows.
\end{proof}

We next consider the effect of Noether normalization applied to $A$,
noting that $k$ is infinite. This result highlights another way that
the present situation is to an extent analogous to that of
Section~\ref{mult}.

We distinguish between the Krull dimension of $A$, denoted $\kdim(A)$,
and the dimension of $A$ as a projective scheme, denoted
$\prdim(A)$. Thus $\kdim(A)=\prdim (A)+1$ (see \cite[pp.~286-287]{e}).

\begin{Proposition}\label{noether}
Let $d=\kdim(A)$. Then there exists a regular sequence of homogeneous
elements of degree $1$ in $A$, which we relabel as $x_{1},\ldots
,x_{d}$, such that $A=k[x_{1},\ldots,x_{n}]$ with $A$ a free finitely
generated graded module over the subring $B:=k[x_{1},...,x_{d}]$, $B$
being isomorphic to the polynomial ring over $k$ in $d$ variables
(with the standard grading) under the natural homogeneous mapping of
degree $0$.
\end{Proposition}

\begin{proof} See \cite[1.5.17(c)]{bh}, together with the
discussion in \cite[\S~3, particularly that on p.~63]{st}. Note also
that $n$ remains invariant under this re-writing of $A$, as
$n=\mu(\cm)=\mu(\cm_{\mathcal{M}})$.
\end{proof}

We now show how to reduce dimension for gr-Shimoda rings.

\begin{Proposition}\label{red-dim}
Let $a$ be a homogeneous element of the gr-Shimoda ring $A$, $a$ lying
in $\cm\setminus \cm^2$. Then $A/aA$ is again a gr-Shimoda ring.
\end{Proposition}

\begin{proof}
First of all we note that such an element $a$ exists. By the graded
version of Nakayama's Lemma, $\cm\neq \cm^{2}$ since $A$ is not a
field. Since $\cm$ is a homogeneous ideal, it has a set of homogeneous
generators not all of which can lie in $\cm^{2}$, by the previous
observation. Pick $a$ to be a suitable member of this set of
generators.

Clearly $A/aA$ is a standard graded ring (with $\kdim(A/aA)$ at least
1) having irrelevant ideal $\cm/aA$. Consider a relevant homogeneous
prime ideal $\fp$ of $A/aA.$ Then $\fp=\fP/aA$ for a unique relevant
homogeneous prime ideal $\fP$ of $A$ that contains $a.$ Since $a\in
\cm\setminus\cm^{2}$, a fortiori $a\in\fP\setminus\cm\fP$. Hence $a$
is a minimal homogeneous generator of $\fP$, so that
\begin{equation*}
\mu (\fp)=\mu (\fP)-1=\height(\fP)-1=\height(\fp),
\end{equation*}
and the result follows.
\end{proof}

We can now give our main result.

\begin{Theorem}\label{shm-graded}
Suppose that in the gr-Shimoda ring $A$, the base field $k$ has
characteristic $0$. Then $\prdim(A)\leq 2$.
\end{Theorem}

\begin{proof}
We suppose that $\prdim (A)\geq 3$ and deduce a contradiction by
finding a homogeneous ideal of height $\prdim (A)-1$ that requires a
generating set of cardinality $\prdim (A)$. By
Proposition~\ref{red-dim} we may suppose that $\prdim (A)=3$.

We use without mention the fact that $A$ is a standard graded
Cohen-Macaulay affine domain, together with the well-known properties
of such domains, such as their being catenary. Note that these are the
only hypotheses (together with the fact that char$(k)=0)$ used in the
remainder of the argument. That is, we prove the following.

\medskip

\noindent \textit{``A standard graded Cohen-Macaulay affine domain
  over a field of characteristic zero and of Krull dimension $4$ has a
  relevant homogeneous prime ideal $\fp$ of height $2$ with
  $\mu(\fp)=3$.''}

\medskip

\noindent {\em Step 1. Constructing the prototype of the prime ideal
  we seek.} Since $\kdim(A)=4$, using Proposition~\ref{noether} we
can choose a regular sequence $x,y,z$ of homogeneous elements of
degree $1$ in $A$.  Set $u=x^{2}-yz$, $v=y^{2}-xz$ and $w=z^{2}-xy$,
and let $I=(u,v,w)A$; note that $u,v,w$ are homogeneous of degree
$2$. It is easily seen that $u,v,w$ are all non-zero elements: for
example, note that $\rad((u,y)A)=\rad((x,y)A)$; hence $(u,y)A$ has
height $2$, so $u\neq 0$.

By way of motivation, we note in passing that $I$ is an almost
complete intersection Northcott ideal (\cite{n}, \cite{v}) and it is
also a determinantal ideal. (The ideal analyzed in Section~\ref{mult}
also had these properties, and they will prove crucial in the analysis
of the ideal $K$ that we will focus on below.) We sketch the details
as follows.

Setting
\begin{equation*}
\Phi =\left(
\begin{array}{rr}
x&z\\-z&-y
\end{array}
\right) ,
\end{equation*}
and letting $^{\top}$ denote matrix transpose, we have that
\begin{equation}\label{transpose}
\Phi\cdot\left(\begin{array}{cc}x&-y\end{array}\right)^{\top}=
  \left(\begin{array}{cc}u&v\end{array}\right)^{\top}
\end{equation}
and that $w=\det(\Phi)$. Note that $u,v$ is a regular sequence,
since $u$ is a non-zero element in the domain $A$ and $v$ cannot lie
in any associated prime of $u$ (which is necessarily at height 1),
because $(u,v,z)A$ has height 3 as a consequence of the equality
$\rad((u,v,z)A)=\rad((x,y,z)A)$. It follows from \cite[\S2]{op},
since $(u,v)A$ has grade 2, that $I$ is grade-unmixed of height 2.
Note also that
\begin{equation*}
I=I_{2}\left(\begin{array}{ccc}
x&y&z\\y&z&x\end{array}
\right) ,
\end{equation*}
so, given the properties of $I$ stated above, the graded version of
the Hilbert-Burch Theorem provides us with the resolution
\begin{equation*}
0\rightarrow A^{2}(-2)\overset{\varphi _{2}}{\rightarrow }A^{3}(-1)
\overset{\varphi _{1}}{\rightarrow }A\rightarrow A/I\rightarrow 0,
\end{equation*}
where
\begin{equation*}
\varphi _{1}=\left(\begin{array}{ccc}w&u&v\end{array}\right) \mbox{
    and }\varphi _{2}^{\top}=\left(\begin{array}{ccc}
    x&y&z\\y&z&x\end{array}\right) ,
\end{equation*}
so this resolution is then minimal. Hence $\mu (I)=3$. We also see,
using the Auslander-Buchsbaum formula, that $I_{\mathcal{M}}$ and
hence $I$ is a Cohen-Macaulay ideal. (For further properties of such
ideals, see \cite[\S2 and 5]{op}.)

We now use and develop \cite[4.6]{suv} along similar lines, in the
present setting.

Consider the polynomial extension $D:=A[X,Y];$ note that $D$ is a
standard graded Cohen-Macaulay domain, on giving each of $X$ and $Y$
the weighting 1. Set
\begin{equation*}
K=(u+xX,v+xY)D:_{D}x.
\end{equation*}
In (\ref{transpose}), on premultiplying by $\adj(\Phi)$, one obtains
the relationship
\begin{equation*}
xw=-yu-zv.
\end{equation*}
One easily checks that, as a result, $w-yX-zY\in K$. We wish to show
that $K$ is a prime ideal of $D$, that $K=L:=(u+xX,v+xY,w-yX-zY)D$ and
that $\mu(K)=3$ (note that $K$ is then a homogeneous ideal of $D$,
being generated by quadrics). The ideal $K$ is the prototype of the
prime ideal in $A$ that we seek so as to establish our contradiction.

We first show that $u+xX$ is a prime element in $D$. Note that, as
remarked above, $u$ is a regular element in the domain $A$, since
$u\neq 0.$ Next, $x$ is regular modulo $uA$, otherwise $(u,x)A$ would
have height $1$ which would contradict the fact that
$\rad((u,x)A)\ (=\rad((x,y)A)\cap \rad((x,z)A))$ has height $2$.
That $u+xX$ is a prime element in $A[X]$ follows from \cite[Ex.~3,
  p.~102]{k}; hence $u+xX$ is also a prime element in $D$. Now
$v+xY\notin (u+xX)D$, since $v\notin uA$, and we deduce that
$u+xX,v+xY$ is a regular sequence in $D$. Note also that $K$ is a
proper ideal in $D$, since otherwise it would follow that $x\in
(u,v)A$, which would contradict the fact that $x$ is actually regular
modulo the larger ideal $I\equiv (u,v,w)A$. Hence $K$ is unmixed of
height $2$, by basic properties of the colon operation vis-\`a-vis
finite intersections and primary ideals. Since $w-yX-zY\in K$, it is
easily seen that, up to radical, $K+xD$ contains the elements
$x,y,z(z-Y)$. Since each of the ideals $(x,y,z)D$ and $(x,y,z-Y)D$ has
height $3$, it follows that $K+xD$ has height $3$ and hence that $x$
is regular modulo $K$, since $K$ is unmixed of height $2$. But it is
clear that $(D/K)_{x}$ is isomorphic to
$A_{x}[\frac{u}{x},\frac{v}{x}]$, i.e., to the domain $A_{x}$. Hence
$D/K$ is a domain, so $K$ is indeed a prime ideal.

Recall the ideal $L:=(u+xX,v+xY,w-yX-zY)D$ introduced above. Let
\begin{equation*}
\Psi=\left(\begin{array}{rr}x+X&-z\\-z+Y&y\end{array}\right) .
\end{equation*}
Note that
$\Psi\cdot\left( \begin{array}{cc}x&y\end{array}\right)^{\top}
  =\left(\begin{array}{cc} u+xX&v+xY\end{array}\right)^{\top}$ 
and that $\det(\Psi)=-(w-yX-zY)$. We have seen above that
$\grade((u+xX,v+xY)D)=2$ and that $K$ (and so $L)$ is a proper
ideal. Hence by \cite[Theorem~2]{n}, the projective dimension of $D/L$
is $2$ and $L$ is grade and hence height unmixed, with all its
associated primes having height 2. The argument above that established
that $K+xD$ had height $3$ actually showed that $L+xD$ had height
3. Hence $x$ is regular modulo $L$.

Now
\begin{equation*}
(u+xX,v+xY)D\subseteq L\subseteq K\equiv (u+xX,v+xY)D:_{D}x,
\end{equation*}
so that on localizing at each associated prime $\fp$ of $L$ we have
\begin{equation*}
(u+xX,v+xY)D_{\mathfrak{p}}\subseteq L_{\mathfrak{p}}\subseteq
  K_{\mathfrak{p}}\equiv (u+xX,v+xY)D_{\mathfrak{p}}.
\end{equation*}
Hence $L_{\mathfrak{p}}=K_{\mathfrak{p}}$ for all such $\fp$, so
$L=K$. Thus
\begin{equation*}
K=(u+xX,v+xY,w-yX-zY)D.
\end{equation*}
So
\begin{equation*}
K=I_{2}\left(\begin{array}{ccc}x&z&-y\\y&x+X&-z+Y\end{array}
\right) ,
\end{equation*}
and again by the Hilbert-Burch Theorem, on setting
\begin{equation*}
\psi_{1}=\left(\begin{array}{ccc}
w-yX-zY&v+xY&-(u+xX)
\end{array}
\right) \mbox{ , }\psi _{2}^{\top}=\left(
\begin{array}{ccc}
x & z & -y \\
y & x+X & -z+Y
\end{array}
\right) ,
\end{equation*}
we have the minimal free resolution
\begin{equation}\label{min-free-res}
0\rightarrow D^{2}(-2)\overset{\psi _{2}}{\rightarrow
}D^{3}(-1)\overset{\psi _{1}}{\rightarrow }D\rightarrow D/K\rightarrow
0
\end{equation}
of $D/K$. In particular, $\mu (K)=3$, and as before, an application of
the Auslander-Buchsbaum formula shows that $K$ is a Cohen-Macaulay
ideal.

So $K$ has the features that we are seeking - a homogeneous prime
ideal of height 2, minimally generated by $3$ elements - except that
the projective variety $\mathbb{V}_{+}(K)$ (consisting
set-theoretically of all relevant homogeneous prime ideals containing
$K$) lies not in $\Proj(A)$ but in $\Proj(A[X,Y])$. We now employ a
double use of a Bertini Theorem to project back into $\Proj(A)$.

\medskip

\noindent {\em Step 2. Retracting back into $A$ via the double use of
  a Bertini Theorem}. For a standard graded Cohen-Macaulay algebra $S$
of positive Krull dimension, note that the scheme $\Proj(S)$ is
integral if and only if $S$ is an integral domain. We use this fact
below without further comment.

In our setting then, the scheme $\Proj(D/K)$ is integral. Note that
pr-$\dim (D/K)=3$. Now apply \cite[(5.5) Satz]{f}, with the role of the
$f_{i}$ there being played by the images of $x_{1},...,x_{n},X,Y$ in
$D/K$, each of which has grade 1. We abuse notation by continuing to
write $x_{1},...,x_{n},X,Y$ for these respective images. Hence, in
applying \cite[(5.5) Satz]{f}, $D_{+}(f_{0},...,f_{n})$ there in fact
equals all of $\Proj(D/K)$. (Recall that $D_{+}(f_{0},...,f_{n})$
consists set-theoretically of all relevant homogeneous prime ideals
that do not contain $\{f_{0},...,f_{n}\}.)$

We now make the elementary observation that if $R$ is a ring and $r$
is an element of $R$, then, for an indeterminate $Z$ over $R$, we have
a natural induced isomorphism $R[Z]/(Z-r)\approx R$ arising from the
natural retraction on $R[Z]$ that maps $Z$ to $r$ (and is the identity
map on $R$).  We apply this observation to a generic hyperplane
section $H$, say, of $\Proj(D/K)$ and then to a further generic
hyperplane section $H^{\prime }$, say, of $\Proj(D/K)\cap H$, resulting
finally in an integral subscheme of $\Proj(A)$.

Hence, by \cite[(5.5) Satz]{f}, for all $\alpha:=(\alpha
_{1},...,\alpha _{n+2})$ in a non-empty Zariski-open set in $k^{n+2}$,
and following the standard notation in \cite[Section 5]{f},
$\mathbb{V}_{+}(f_{\alpha})$ is an integral subscheme of
$\Proj(D/K)$, where $f_{\alpha}:=\alpha _{1}x_{1}+...+\alpha
_{n}x_{n}+\alpha _{n+1}X+\alpha _{n+2}Y$. Without loss of generality,
we may suppose that $\alpha _{n+2}\neq 0$, using the intersection of
non-empty Zariski-open sets in $k^{n+2}$. Note that $K$ is generated
by quadrics, so that the variety $\mathbb{V}_{+}(K)$ is nondegenerate
(i.e., is not contained in a hyperplane). Set
\begin{equation*}
Y^{\prime }=-(\alpha _{1}x_{1}+...+\alpha _{n}x_{n}+\alpha
_{n+1}X)/\alpha _{n+2}\mbox{ , }D^{\prime }=A[X].
\end{equation*}
By our elementary observation therefore, $K^{\prime }:=(u+xX,v+xY^{\prime
},w-yX-zY^{\prime })D^{\prime }$ is a homogeneous prime ideal which is
generated by quadrics, and pr-$\dim (D^{\prime }/K^{\prime })=2$.

Fix such an $\alpha$ in $k^{n+2}$. Proceeding as before then, there
exists a non-empty Zariski-open set in $k^{n+1}$, depending on
$\alpha$, such that for all $\beta:=(\beta _{1},...,\beta _{n+1})$ in
this set, we have $\beta _{n+1}\neq 0$, and on setting
\begin{equation*}
X^{\prime }=-(\beta _{1}x_{1}+...+\beta _{n}x_{n})/\beta _{n+1}\mbox{
  , and }Y^{\prime \prime }=-(\alpha _{1}x_{1}+...+\alpha
_{n}x_{n}+\alpha _{n+1}X^{\prime })/\alpha _{n+2},
\end{equation*}
$K^{\prime \prime }:=(u+xX^{\prime },v+xY^{\prime \prime
},w-yX^{\prime }-zY^{\prime \prime })A$ is a prime ideal with
$\prdim(A/K^{\prime \prime })=1$, so $\grade(K^{\prime \prime
})=\height(K^{\prime \prime })=2$. Now the presentation
\begin{equation*}
K=I_{2}\left(
\begin{array}{ccc}
x & z & -y \\
y & x+X & -z+Y
\end{array}
\right)
\end{equation*}
carries over to give the equality
\begin{equation*}
K^{\prime \prime }=I_{2}\left(
\begin{array}{ccc}
x & z & -y \\
y & x+X^{\prime } & -z+Y^{\prime \prime }
\end{array}
\right) ,
\end{equation*}
and hence the resolution (\ref{min-free-res}) carries over to give
the minimal presentation
\begin{equation*}
0\rightarrow A^{2}(-2)\overset{\psi_{2}^{\prime }}{\rightarrow
}A^{3}(-1)\overset{\psi_{1}^{\prime }}{\rightarrow }A\rightarrow
A/K^{\prime \prime }\rightarrow 0,
\end{equation*}
where
\begin{equation*}
\psi_{1}^{\prime }=\left(
\begin{array}{ccc}
w-yX^{\prime }-zY^{\prime \prime }&v+xY^{\prime \prime }&-(u+xX^{\prime})
\end{array}\right) \mbox{ , }(\psi_{2}^{\prime })^{\top}=\left(
\begin{array}{ccc}
x & z & -y \\ y & x+X^{\prime } & -z+Y^{\prime \prime }
\end{array}
\right) .
\end{equation*}

In particular, $\mu (K^{\prime \prime })=3$, so $K^{\prime \prime }$
is the prime ideal we seek, and this yields the desired contradiction.
\end{proof}

\begin{Remark}{\rm 
Note that this result is optimal. Consider the ring $A$ where $A$ is
the polynomial ring $k[{X,Y,Z]}$, with $k$ algebraically closed, under
the natural grading. We see from the projective Nullstellensatz and
the fact that $A$ is a UFD that $A$ is a gr-Shimoda ring, and
pr-dim$(A)=2$.  }\end{Remark}

\section{More technical results in the standard graded 
case}\label{viamiller}

We can adapt arguments used in M. Miller's paper \cite{miller} to
prove the following results. We provide only some additional details,
as necessary.

\begin{Theorem}\label{approachmil} (cf. \cite[Corollary~2.7 
(and Theorem~2.1)]{miller}).  Let $S$ be a standard graded algebra
  over a field $k$ of characteristic $0$, such that $S$ is a domain,
  $\kdim(S)\geq 4$ and $S$ satisfies the Serre conditions $(R_{2})$
  and $(S_{3})$. Suppose further that homogeneous prime ideals of $S$
  of height $1$ are principal. Then $S$ possesses height $2$
  homogeneous prime ideals requiring an arbitrarily large number of
  generators.
\end{Theorem}

\begin{proof}[Sketch of Proof] Let $M$ be the graded second syzygy in a
minimal graded free resolution of a finitely generated graded
$S$-module of finite projective dimension at least 3, so that $M$
itself is not free. Note that by the Evans-Griffith Syzygy Theorem
(cf. \cite[Corollary~9.5.6 and subsequent remark]{bh}) $\rank(M)\geq
2$.  Let $d=\rank(M)-1$.

Suppose that we have the graded presentation
$S^{c}\buildrel\alpha\over\rightarrow S^{b}\rightarrow M\rightarrow 0$
(ignoring twists). Set $A=S[Y_{i,j}\mid 1\leq i\leq b,1\leq j\leq d]$,
giving each indeterminate $Y_{i,j}$ over $S$ weight 1. Define
$\phi:A^{d}\rightarrow A^{b}$ by the matrix
$\left(Y_{i,j}\right)$. Let $J=\coker(\alpha\oplus\phi)$. Hence we
have the presentation
\begin{eqnarray*}
A^{c}\oplus A^{d}\overset{\alpha \oplus \phi }{\rightarrow
}A^{b}\rightarrow J\rightarrow 0.
\end{eqnarray*}
Clearly $\rank(J)=1$. Note that $N:=M\otimes A\simeq A^{b}/\alpha
(A^{c})$ as graded modules, so there is an exact sequence
\begin{eqnarray*}
0\rightarrow A^{d}\overset{\psi }{\rightarrow }N\rightarrow
J\rightarrow 0,
\end{eqnarray*}
with the graded homomorphism $\psi$ induced by $\phi$.

A straightforward adaptation of Miller's argument ([op. cit., p.~31])
then shows that the graded module $J$ is a torsion-free $A$-module,
and so, as is easily seen, a homogeneous ideal (up to graded
isomorphism).

Next we show that we may suppose that the homogeneous ideal $J$ has
height 2. Suppose that $J$ has height 1. By an obvious `homogeneous'
analogue of the proof of \cite[Theorem~5(b)]{k}, given our hypotheses,
a non-zero homogeneous element of $A$ that is not a unit is a unique
product of homogeneous elements each of which generates a homogeneous
prime ideal (we note that, since $A$ is a domain, a product of
non-zero elements of $A$ is homogeneous if and only if each element is
homogeneous). Hence, taking a finite set of homogeneous generators of
$J$, we can find a smallest homogeneous principal ideal $aA$, with $a$
a homogeneous element, that contains $J$. Clearly $a^{-1}J$ is a
homogeneous ideal. We claim that $a^{-1}J$ is of height at least
$2$. For otherwise, $a^{-1}J$ is contained in a homogeneous prime $P$
of height $1$, and by hypothesis, $P=bA$ for some homogeneous element
$b\in A$; it would then follow that $J\subseteq abA\subseteq aA$, so
$abA=aA$, i.e., $bA=A$, a contradiction. Hence we may suppose that $J$
has height at least $2$. Another straightforward adaptation of
Miller's argument ([op. cit., pp.~31-32]) shows that $J$ has
precisely height $2$ and that $J$ is in fact prime.

Finally, on replacing $A$ by its localization at its irrelevant
maximal ideal $\mathcal{M}$, we can directly apply the argument in
\cite{miller} to prove his Corollary 2.7 with reference to the ideal
$\overline{J}:=J_{\mathcal{M}}+(Z_{1},\ldots,Z_{bd})A_{\mathcal{M}}$
considered in $S_{\mathcal{N}},$ where $\mathcal{N}$ denotes the
irrelevant maximal ideal in $S$ (cf. [op. cit., p.~33 and top of
  p.~34]): note that if $x_{1},\ldots ,x_{n}$ are the degree 1
generators of $\mathcal{N}$, then the $Z_{l}$ are general $k$-linear
combinations of $x_{1},\ldots,x_{n}$ and the $Y_{i,j}$, and so are
homogeneous forms of degree $1$. We next observe that a standard
graded algebra $D$ with irrelevant maximal ideal $\mathcal{P}$ is a
domain if $D_{\mathcal{P}}$ is a domain. It follows that the
homogeneous ideal $\widetilde{J}:=J+(Z_{1},\ldots,Z_{bd})A$,
considered in $S$, is a prime ideal. Since $\widetilde{J}$ requires at
least as many generators as $\overline{J}$, the result follows.
\end{proof}

The arguments in the proof of Step~1 of the proof of
Theorem~\ref{shm-graded} can easily be adapted for use in the
following more general situation. We adopt the notation of
Section~\ref{sgraded}.

\begin{Theorem}\label{variant}
Let $A$ be a standard graded algebra over a field of characteristic
$0$, with $A$ a domain. Suppose that $\kdim(A)\geq 4$ and that $A$
satisfies the Serre condition $(S_{3})$. Then $A$ possesses a
homogeneous prime ideal $P$ of height $2$ with $\mu(P)=3$.
\end{Theorem}

\begin{proof}[Sketch of Proof] 
In the light of \cite[Propositions~1.5.15(e), 1.5.11 and 1.5.12]{bh},
a straightforward adaptation of the original argument means that we
need only comment on how to adapt the details of Step 2, viz. the
application of the double use of a Bertini Theorem, in order to finish
the proof.

Replace $D$ by its localization at its irrelevant ideal,
$\mathcal{M}^{\prime}$ say, and $A$ by its localization at its
irrelevant ideal $\mathcal{M}$.

As before, $D/K$ has projective dimension $2$. Since $A$ satisfies the
condition $(S_{3})$ and $X,Y$ are indeterminates over $A$,
$\depth(D)\geq 5$. By the Auslander-Buchsbaum Formula therefore,
$\depth(D/K)\geq 3$. We can now mimic the argument in Part~2 of the
proof of \cite[Theorem~2.1, p.~33ff.]{miller}, here cutting with 2
generic hyperplanes $ Z_{1}$ and $Z_{2}$, in the notation on p.~33 of
\cite{miller}. 

By H. Flenner's Bertini Theorem (cf. \cite[p.~32]{miller}),
$D/(K+Z_{1}D)$ is analytically irreducible. Note that
$\depth(D/(K+Z_{1}D))\geq 2$. Hence another application of Bertini's
Theorem is allowed and we have that $D/(K+(Z_{1},Z_{2})D)$ is again
analytically irreducible.

From the short exact sequence $0\rightarrow
D\overset{Z_{1}}{\rightarrow } D\rightarrow D/Z_{1}D\rightarrow 0$, we
quickly establish that $\Tor_{i}^{D}(D/K,D/Z_{1}D)=0$ for $i\geq 1$.
Hence the minimal free resolution of $D/K$ over $D$ afforded by the
Hilbert-Burch Theorem descends via tensoring with $D/Z_{1}D$ to give
the minimal free resolution over $D/Z_{1}D$ of $D/(K+Z_{1}D)$, this
being in Hilbert-Burch form. Repeating the argument on further
factoring out $Z_{2}\cdot D/Z_{1}D$ shows that
\begin{eqnarray*}
\mu _{D/(Z_{1},Z_{2})D}((K+(Z_{1},Z_{2})D)/(Z_{1},Z_{2})D)=3,
\end{eqnarray*}
and the full result now follows.
\end{proof}

\begin{Remark}{\rm 
Suppose we have a standard graded domain $S$ over a field $k$ such
that every homogeneous prime ideal of height 1, respectively 2, is
generated by 1, respectively 2, elements. It easily follows from the
first part of the proof of \cite[Theorem~1.5.9]{bh} that $S$ satisfies
the condition $(R_{2})$ . Hence if we further suppose that
$\prdim(S)\geq 3$ and that $k$ has characteristic 0, then $S$
satisfies the hypotheses of the Theorem~\ref{approachmil}. It
follows that Theorem~\ref{shm-graded} is a consequence of 
Theorem~\ref{approachmil}.

On the other hand, it is clear that Theorem~\ref{shm-graded} is also a
consequence of Theorem~\ref{variant}.}\end{Remark}

\section*{Acknowledgment} 
We would like to thank J.-M. Giral for drawing the
paper \cite{miller} to our attention and to thank Abd\'o Roig for
useful discussions.

\vspace{0.5cm}

\parbox[t]{6.5cm}{\footnotesize

\noindent Shiro Goto

\noindent Department of Mathematics

\noindent School of Science and Technology

\noindent Meiji University

\noindent Tama, Kawasaki, KANAG 214, JAPAN

\noindent goto@math.meiji.ac.jp}
\,
\parbox[t]{7cm}{\footnotesize

\noindent Liam O'Carroll

\noindent Maxwell Institute for Mathematical Sciences

\noindent School of Mathematics

\noindent University of Edinburgh

\noindent EH9 3JZ, Edinburgh, Scotland

\noindent L.O'Carroll@ed.ac.uk } 

\vspace{0.5cm}

\parbox[t]{6.5cm}{\footnotesize

\noindent Francesc Planas-Vilanova

\noindent Departament de Matem\`atica Aplicada~1

\noindent Universitat Polit\`ecnica de Catalunya

\noindent Diagonal 647, ETSEIB

\noindent 08028 Barcelona, Catalunya

\noindent francesc.planas@upc.edu}

{\small
\begin{thebibliography}{cc}
\bibitem[BH]{bh}{W. Bruns, J. Herzog, Cohen-Macaulay rings, Cambridge
  studies in advanced mathematics 39, Cambridge University Press
  1993.}
\bibitem[Dav]{da}{E.D. Davis, Ideals of the principal class,
  $R$-sequences and a certain monoidal transformation. Pacific
  J. Math. {\bf 20} (1967), 197-205.}
\bibitem[Die]{di}{J. Dieudonn\'e, Topics in local algebra. Notre Dame
  Mathematical Lectures, No. 10 University of Notre Dame Press, Notre
  Dame, Ind. 1967.}
\bibitem[Eis]{e}{D. Eisenbud, Commutative algebra with a view toward
  algebraic geometry. Graduate Texts in Mathematics,
  150. Springer-Verlag, New York, 1995.}
\bibitem[Fle]{f}{H. Flenner, Die S\"atze von Bertini f\"ur lokale
  Ringe. Math. Ann. {\bf 229} (1977), no. 2, 97-111.}
\bibitem[HIO]{hio}{M. Herrmann, S. Ikeda, U. Orbanz, Equimultiplicity
  and Blowing up. An Algebraic Study. With an Appendix by
  B. Moonen. Springer-Verlag, Berlin, 1988.}
\bibitem[Kap]{k}{I. Kaplansky, Commutative rings. Revised edition. The
  University of Chicago Press, Chicago, Ill.-London, 1974.}
\bibitem[KS]{ks}{K. Kiyek, J. St\"uckrad, Integral closure of monomial
  ideals on regular sequences, Rev. Mat. Iberoamericana {\bf 19}
  (2003), 483-508.}
\bibitem[Mat1]{m1}{H. Matsumura, Commutative algebra. Second
  edition. Mathematics Lecture Note Series, 56. Benjamin/Cummings
  Publishing Co., Inc., Reading, Mass., 1980.}
\bibitem[Mat2]{m2}{H. Matsumura, Commutative ring theory. Translated
  from the Japanese by M. Reid. Cambridge Studies in Advanced
  Mathematics, 8. Cambridge University Press, Cambridge, 1986.}
\bibitem[Mil]{miller}{M. Miller, Bourbaki's theorem and prime
  ideals. J. Algebra {\bf 64} (1980), 29-36.}
\bibitem[Nor]{n}{D.G. Northcott, A homological investigation of a
  certain residual ideal, Math. Ann. {\bf 150} (1963), 99-110.}
\bibitem[OP]{op}{L. O'Carroll, F. Planas-Vilanova, Ideals of
  Herzog-Northcott type. Proc. Edinb. Math. Soc. (2) {\bf 54} (2011),
  no. 1, 161-186.}
\bibitem[Sal]{sa}{J.D. Sally, Numbers of generators of ideals in local
  rings, Lecture Notes in Pure and Applied Mathematics {\bf 35},
  Marcel Dekker, Inc., New York and Basel, 1978.}
\bibitem[Shi]{shi}{Y. Shimoda, On power stable ideals and symbolic
  power ideals, The Proceedings of the $3^{\mbox{rd}}$ Japan Vietnam
  Joint Seminar on Commutative Algebra, 90-95, Hanoi (Vietnam),
  December 2-7, 2007.}
\bibitem[Sta]{st}{R. P. Stanley, Hilbert functions of graded
  algebras. Advances in Math. {\bf 28} (1978), no. 1, 57-83.}
\bibitem[SUV]{suv}{A. Simis, B. Ulrich, W.V. Vasconcelos, Jacobian
  dual fibrations. Amer. J. Math. {\bf 115} (1993), no. 1, 47-75.}
\bibitem[SH]{sh}{I. Swanson, C. Huneke, Integral closure of ideals,
  rings, and modules. London Mathematical Society Lecture Note Series,
  336. Cambridge University Press, Cambridge, 2006.}
\bibitem[VV]{vv}{P. Valabrega, G. Valla, Form rings and regular
  sequences. Nagoya Math. J. {\bf 72} (1978), 93-101.}
\bibitem[Vas]{v}{W.V. Vasconcelos, Computational methods in commutative
  algebra and algebraic geometry. Algorithms and Computation in
  Mathematics, 2. Springer-Verlag, Berlin, 1998.}
\end{thebibliography}}
\end{document}